\documentclass[11pt]{amsart}

\usepackage[utf8]{inputenc}
\usepackage[margin=1in]{geometry}
\usepackage[titletoc,title]{appendix}
\usepackage{pifont}
\usepackage{yfonts}
\RequirePackage[dvipsnames,usenames]{xcolor}
\usepackage[colorlinks=true,linkcolor=blue]{hyperref}%
\hypersetup{
    colorlinks=true,
    linkcolor=RoyalBlue,
    filecolor=magenta,      
    urlcolor=RoyalBlue,
    pdfpagemode=FullScreen,
    citecolor=BurntOrange
}
\usepackage[alphabetic,lite]{amsrefs}

\usepackage{amsmath,amsfonts,amssymb,mathtools}
\usepackage{enumitem}
\usepackage{graphicx,float}
\usepackage{tikz-cd}
\usepackage[ruled,vlined]{algorithm2e}
\usepackage{algorithmic}
\usepackage{amsthm}
\usepackage{mathrsfs} 

\usepackage{ytableau}
\usepackage{comment}
\usepackage{bigints}
\usepackage{setspace}
\numberwithin{equation}{section}

\newtheorem{theorem}{Theorem}[section]
\newtheorem{corollary}[theorem]{Corollary}
\newtheorem{lemma}[theorem]{Lemma}

\newtheorem{question}[theorem]{Question}
\newtheorem{remark}[theorem]{Remark}
\newtheorem{example}[theorem]{Example}

\newcommand{\frakm}{\mathfrak{m}}

\newcommand{\CC}{\mathbb{C}}
\newcommand{\ZZ}{\mathbb{Z}}
\newcommand{\GG}{\mathbb{G}}
\newcommand{\DD}{\mathcal{D}}
\newcommand{\D}{\mathcal{D}}
\newcommand{\Q}{\mathcal{Q}}
\newcommand{\mc}{\mathcal}

\newcommand{\Schur}{\mathbb{S}}
\DeclareMathOperator{\Soc}{Socle}
\DeclareMathOperator{\Hom}{Hom}
\DeclareMathOperator{\Ext}{Ext}
\DeclareMathOperator{\Tor}{Tor}
\DeclareMathOperator{\Dim}{dim}

\DeclareMathOperator{\GL}{GL}
\DeclareMathOperator{\Sym}{Sym}
\DeclareMathOperator{\Pf}{Pf}
\DeclareMathOperator{\Det}{det}
\DeclareMathOperator{\Dom}{dom}
\DeclareMathOperator{\Grass}{Gr}
\DeclareMathOperator{\Sort}{sort}

\newcommand{\defi}[1]{{\upshape\sffamily #1}}

\newcommand{\jiamin}[1]{{\color{cyan} \sf $\clubsuit\clubsuit\clubsuit$ Jiamin: [#1]}}
\newcommand{\mike}[1]{{\color{red} \sf $\clubsuit\clubsuit\clubsuit$ Mike: [#1]}}

\makeatletter
\@namedef{subjclassname@2020}{%
  \textup{2020} Mathematics Subject Classification}
\makeatother

\title[Socle degrees for local cohomology modules]{Socle degrees for local cohomology modules of thickenings of maximal minors and sub-maximal Pfaffians}
\author{Jiamin Li}

\author{Michael Perlman}

\subjclass[2020]{}

\begin{document}

\maketitle

\begin{abstract}
	Let $S$ be the polynomial ring on the space of non-square generic matrices or the space of odd-sized skew-symmetric matrices, and let $I$ be the determinantal ideal of maximal minors or $\Pf$ the ideal of sub-maximal Pfaffians, respectively. Using desingularizations and representation theory of the general linear group we expand upon work of Raicu--Weyman--Witt to determine the $S$-module structures of $\Ext^j_S(S/I^t, S)$ and $\Ext^j_S(S/\Pf^t, S)$, from which we get the degrees of generators of these $\Ext$ modules. As a consequence, via graded local duality we answer a question of Wenliang Zhang on the socle degrees of local cohomology modules of the form $H^j_\frakm(S/I^t)$.
\end{abstract}

\section{Introduction}
Given a polynomial ring $S$ over a field, $I$ a homogeneous ideal, and $\mathfrak{m}$ the homogeneous maximal ideal, local cohomology modules of the form $H^j_\frakm(S/I^t)$ are generally not finitely generated, and in particular may be nonzero in arbitrarily negative internal degrees. In the recent past, asymptotic behavior (as $t$ goes to infinity) of these modules has received much attention (see, for example, \cite{Bhatt+2, DaoMontano, Bhatt+}). Inspired by this past work, Wenliang Zhang \cite{Zhang} posed the following question, where $\Soc(M) = (0 :_M \frakm)$ denotes the socle of an $S$-module $M$.
     
    \begin{question}\label{zhang}
	Let $I$ be homogeneous and $j\geq 0$. Does there exist a constant $c=c(j,I)$ such that $$\Soc(H^j_\frakm(S/I^t))_{d} = 0,$$ 
	for all $t>0$ and $d\leq ct$?
    \end{question}

    Under appropriate assumptions on the ideal $I$ (e.g., local complete intersection, specific cohomological indices), Zhang answers Question \ref{zhang} in the positive (see \cite[Theorem 1.2, 1.3, 1.4]{Zhang}). Question \ref{zhang} remains open in the case of determinantal ideals and ideals of Pfaffians.  
    
In this paper, we study socle degrees for local cohomology modules of powers of ideals of maximal minors of a generic matrix, and ideals of sub-maximal Pfaffians of an odd-sized skew-symmetric matrix. More precisely, we work in the following two settings:
\begin{enumerate}
    \item $S=\Sym(\CC^m\otimes \CC^n)$ is the polynomial ring on the space of $m\times n$ generic matrices ($m>n$), endowed with the action of $\GL_m(\CC)\times \GL_n(\CC)$, and $I$ is the ideal of $n\times n$ minors.

    \item $S=\Sym(\bigwedge^2 \CC^m)$ is the polynomial ring on the space of $m\times m$ skew-symmetric matrices ($m=2n+1$), endowed with the action of $\GL_{m}(\CC)$, and $\Pf$ is the ideal of $2n\times 2n$ Pfaffians,
\end{enumerate}
where for an even-sized skew-symmetric matrix $A$, the Pfaffian is $\sqrt{\operatorname{det}(A)}$ (see \cite[3.4]{BrunsHerzog}). 

Our first result below is concerned with degrees of generators of modules of the form $\Ext^j_S(S/J^t,S)$ for $J=I$ and $J=\Pf$, which are related to the socle degrees in question via graded local duality. To clarify the statement of the theorem, we first explain when these modules are nonzero. By \cite[Theorem 4.3, Theorem 5.3]{RaicuWeymanWitt} we have
\begin{enumerate}
\item $\Ext^j_S(S/I^t,S)\neq 0$ if and only if $j=s(m-n)+1$ for some $1\leq s\leq n$, and $t\geq s$,
\item $\Ext^j_S(S/\Pf^t,S)\neq 0$ if and only if $j=2s+1$ for some $1\leq s\leq n$, and $t\geq 2s-1$.
\end{enumerate}

In particular, in both cases, for a fixed power $t>0$, we have that $\Ext$ is nonzero in at most $n$ cohomological degrees, and for $t\geq n$ (resp. $t\geq 2n-1$) it is nonzero in exactly $n$ degrees.

\begin{theorem}\label{main} Let $1\leq s\leq n$ and $t>0$. The following is true about the nonzero Ext modules:
\begin{enumerate}
    \item Let $I$ be the ideal of maximal minors of a generic matrix. The module $\Ext^{s(m-n)+1}_S(S/I^t, S)$ is generated in a single degree, which is 
    $$
    -s(t+m-s).
    $$
    \item Let $\Pf$ be the ideal of sub-maximal Pfaffians. The module $\Ext^{2s+1}_S(S/\Pf^t, S)$ is generated in a single degree, which is 
    $$
    -s(t+2(n-s+1)).
    $$
\end{enumerate}
\end{theorem}

The structure of the above modules as representations of $\GL_m(\CC)\times \GL_n(\CC)$ (resp. $\GL_{m}(\CC)$) was calculated in \cite{RaicuWeymanWitt} and is recalled below. Theorem \ref{main} is a consequence of our main results Theorem \ref{degrees_gens_ext_generic} and Theorem \ref{degrees_gens_ext_skew}, which explicitly describe the structure of these representations as $S$-modules. Theorem \ref{main} follows from the fact that $\Ext^{s(m-n)+1}_S(S/I^t, S)$ and $\Ext^{2s+1}_S(S/\Pf^t, S)$ are both generated by a single irreducible representation. 

As a corollary, we provide a complete characterization of cyclicity of these $\Ext$ modules.   
\begin{corollary}\label{main_corollary}
With the notation above, the following is true about the nonzero Ext modules:
\begin{enumerate}
    \item The module $\Ext^{s(m-n)+1}_S(S/I^t, S)$ is cyclic if and only if $s=t=n$.
    
    \item The module $\Ext^{2s+1}_S(S/\Pf^t, S)$ is cyclic if and only if $t=2s-1$.
\end{enumerate}
\end{corollary}

In particular, $\Ext^{s(m-n)+1}_S(S/I^t, S)$ is cyclic for a unique choice of $s$ and $t$, whereas in the skew-symmetric case we have that for all $1\leq s \leq n$ there exists a unique $t$ for which $\Ext^{2s+1}_S(S/\Pf^t, S)$ is cyclic. The reason for the different behavior in the two cases may be explained using the underlying representation theory. The module $\Ext^{s(m-n)+1}_S(S/I^t, S)$ is cyclic if and only if it is generated by a one-dimensional (and thus $\operatorname{SL}_m(\CC)\times \operatorname{SL}_n(\CC)$-invariant) sub-representation. However, $\Ext^{s(m-n)+1}_S(S/I^t, S)$ has no $\operatorname{SL}_m(\CC)\times \operatorname{SL}_n(\CC)$-invariants unless $s=n$. On the other hand, for all $s$ and $t\geq 2s-1$ we have that $\Ext^{2s+1}_S(S/\Pf^t, S)$ contains a (unique) $\operatorname{SL}_{2n+1}(\CC)$-invariant subspace.

When $s=n$ the Ext modules are finite-dimensional vector spaces (in both cases). We exhibit one such example in the case of generic matrices. We carry out infinite-dimensional examples in Sections \ref{cal_deg_socle_generic} and \ref{cal_deg_socle_skew}.

\begin{example}\label{Ex}
Let $m=4$, $n=2$, and $s=n$, in which case we are considering cohomological degree $5$. By \cite[Theorem 4.3]{RaicuWeymanWitt}, we have that
$$
\Ext^5_S(S/I^t, S)\neq 0\quad \iff \quad t\geq 2,
$$
and there is an infinite chain of inclusions
$$
\Ext^5_S(S/I^2, S)\subseteq \Ext^5_S(S/I^3, S)\subseteq \Ext^5_S(S/I^4, S)\subseteq  \cdots,
$$
and the module $\Ext^5_S(S/I^4, S)$ has the following decomposition into irreducible representations:
\begin{align*}
\Ext^5_S(S/I^4, S) & = (\Schur_{(-2^2, -4^2)}\CC^4 \otimes \Schur_{(-6^2)}\CC^2) \oplus (\Schur_{(-2^2,-3,-4)}\CC^4 \otimes \Schur_{(-5, -6)}\CC^2) \\
&\oplus (\Schur_{(-2^2, -3^2)}\CC^4 \otimes \Schur_{(-5^2)}\CC^2) \oplus (\Schur_{(-2^3, -3)}\CC^4 \otimes \Schur_{(-4,-5)}\CC^2)\\
&\oplus (\Schur_{(-2^4)}\CC^4 \otimes \Schur_{(-4^2)}\CC^2),
\end{align*}
where $\Schur_\lambda(-)$ is a Schur functor (see Section 2.1).

It follows from Theorem \ref{degrees_gens_ext_generic} below that $\Schur_{(-2^2, -4^2)}\CC^4 \otimes \Schur_{(-6^2)}\CC^2$ is the space of generators of $\Ext^5_S(S/I^4, S)$, with degree $-12$. The second and the third lines above are the decomposition of $\Ext^5_S(S/I^3,S)$, which is generated by $\Schur_{(-2^2,-3^2)}\CC^4 \otimes \Schur_{(-5^2)}\CC^2$. The third line is the decomposition of $\Ext^5_S(S/I^2, S)$, which is cyclic with generator  $\Schur_{(-2^4)}\CC^4 \otimes \Schur_{(-4^2)}\CC^2$.
\end{example}

After applying graded local duality (see Section \ref{localDualSection}), we obtain the following positive answer to Question \ref{zhang}. In fact, we exhibit the socle degree as an explicit linear function in $t$. 
\begin{corollary}\label{CorSocle}
Let $1\leq s\leq n$ and $t>0$. Using notation as above, we have the following:
\begin{enumerate}
\item The  degree of minimal generators of $\Soc(H^{mn-(m-n)s-1}_\frakm(S/I^t))$ is 
$$
s(t+m-s)-mn.
$$

\item The degree of minimal generators of $\Soc(H^{2n^2+n - 2s-1}_\frakm(S/\Pf^t))$ is
$$
s(t+2(n-s+1)) - (2n^2+n).
$$
\end{enumerate}
\end{corollary}

Following \cite{RaicuWeymanWitt}, several authors have investigated finiteness properties of the local cohomology modules $H^j_\frakm(S/I^t)$ and $H^j_\frakm(S/\Pf^t)$ \cite{Kenkel, Li, Simper}. In particular, in \cite{Li}, it is established that, for $J=I$ or $J=\Pf$, for all $t>0$, none of the modules $H^{j}_\frakm(S/J^t)$ have finite length except for $H^{n^2-1}_\frakm(S/J^t)$ when $J=I$ and $H^{2n^2-n-1}_\frakm(S/J^t)$ when $J=\Pf$ (the cases $s=n$ above). Moreover, the asymptotic behaviors of the length of $H^{n^2-1}_\frakm(S/I^t)$ and $H^{2n^2-n-1}_\frakm(S/\Pf^t)$ are studied in detail (see \cite[Theorem 1.1, Theorem 1.2]{Li}). See also \cite{Raicu18} (resp. \cite{Perlman}) for Kodaira-like vanishing result on the graded pieces of $H^{j}_\frakm(S/I^t)$ (resp. $H^{j}_\frakm(S/\Pf^t)$). More recently, Simper \cite{Simper} focuses on the case $s=n$ for generic matrices, and determines the annihilators of the modules $\Ext^{(m-n)n+1}_S(S/I^t, S)$, and goes on to show that $H^{n^2-1}_\frakm(S/I^t)$ is cyclic. In the case of $m\times (m-1)$-matrices he explicitly realizes the natural inclusions of the modules $\Ext^{m}_S(S/I^t, S)$ into $H^{m}_I(S)$. Our results complement the previous study on this topic mentioned above, in that we answer Question \ref{zhang} for all cohomological degrees, and for ideals of sub-maximal Pfaffians.

All of the above results are consequences of the following main result. Let $J=I$ or $J=\Pf$. For all $j\geq 0$, the $\GL$ structure of $H^{j}_J(S)$ is determined by a subset of the set of dominant weights (see Sections \ref{sectionDgen} and \ref{DmodSkew}), and for a dominant weight $\lambda$, we write $H^{j}_J(S)_{\lambda}$ for the corresponding isotypic component. We order the set of dominant weights by saying $\lambda\leq \mu$ if $\lambda_i\leq \mu_i$ for all $i=1,\cdots, n$. 

\begin{theorem}[Theorem \ref{S_structure_Dp}, Theorem \ref{S_structure_Bp}]
Let $j\geq 0$, let $\lambda, \mu$ be dominant weights, and let $J=I$ or $J=\Pf$. We have that $H^{j}_J(S)_{\lambda}$ generates $H^{j}_J(S)_{\mu}$ over $S$ if and only if $\lambda\leq \mu$.
\end{theorem}

We summarize the proof of Theorem \ref{main} and describe the organization of the article. 

\begin{enumerate}
    \item For $j\geq 0$, the direct limit over $t$ of the modules $\Ext^{j}_S(S/J^t, S)$ is the local cohomology module $H^j_J(S)$. When $J$ is the ideal of maximal minors or sub-maximal Pfaffians, this limit is a union, i.e. for all $t>0$ and $j\geq 0$ we have an inclusion:
    $$
    \Ext^{j}_S(S/J^t, S)\subseteq H^j_J(S).
    $$
   
    \item Each nonzero local cohomology module $H^j_J(S)$ is a simple $\mathcal{D}$-module, where $\mathcal{D}$ is the Weyl algebra of differential operators associated to $S$. In Sections \ref{genericModuleSection} and \ref{Bgens}, we use desingularizations and Bott's Theorem for Grassmannians to determine the $S$-module structure of these simple $\mathcal{D}$-modules in terms of representation theory of the general linear group. In Section \ref{prelim}, we review the necessary background on representation theory and Bott's Theorem, and in Sections \ref{sectionDgen} and \ref{DmodSkew}, we recall basics of the relevant simple $\mathcal{D}$-modules.
     \item  In Sections \ref{cal_deg_socle_generic} and \ref{cal_deg_socle_skew}, we use (1) and (2) to deduce the $S$-module structure of $\Ext^{j}_S(S/J^t, S)$ from that of $H^j_J(S)$. Via graded local duality (recalled in Section \ref{prelim}), we obtain Corollary \ref{CorSocle} on the socle degrees. 
\end{enumerate}

\section{Preliminaries}\label{prelim}

In this section, we recall the necessary background on representation theory, cohomology of vector bundles on Grassmannians, relative spaces of matrices, and graded local duality.

\subsection{Representation Theory}\label{prelimRep}
Given a positive integer $N$, and an $N$-dimensional complex vector space $V$, we write $\GL(V)$ for the general linear group of linear automorphisms of $V$. The irreducible representations of $\GL(V)$ are indexed by \defi{dominant weights} $\lambda=(\lambda_1\geq \lambda_2\geq \cdots \geq \lambda_N)\in \mathbb{Z}^N$. We write $\mathbb{Z}^N_{\textnormal{dom}}$ for the set of dominant weights, and we write $\mathbb{S}_{\lambda}V$ for the corresponding irreducible representation, where $\mathbb{S}_{\lambda}(-)$ is a \defi{Schur functor} \cite[Chapter 6]{FultonHarris}. We abbreviate $d$ repeated entries $(a,\cdots, a)$ to $(a^d)$. For example, we have $\mathbb{S}_{(d)}V=\operatorname{Sym}^d(V)$, and for $\lambda=(1,\cdots,1)=(1^d)$ we have $\mathbb{S}_{\lambda}V=\bigwedge^dV$. The dimension of $\mathbb{S}_{\lambda}V$ is given by the \defi{hook-content formula} \cite[Theorem 6.3]{FultonHarris} 
\begin{equation}\label{dim_schur}
\Dim_\CC(\Schur_\lambda V) = \prod_{N\geq j>i \geq 1} \dfrac{\lambda_i - \lambda_j + j - i}{j-i}.    
\end{equation}
For instance, $\Schur_\lambda V$ is one-dimensional if and only if $\lambda_1=\lambda_2=\cdots =\lambda_N$.

We endow the dominant weights with the following partial order. Given $\lambda, \mu \in \mathbb{Z}^N_{\textnormal{dom}}$ we have
$$
\lambda \geq \mu \quad \iff \quad \lambda_i\geq \mu_i\quad\textnormal{for $i=1,\cdots,N$.}
$$

\subsection{Grassmannians and Bott's Theorem}

Given an $N$-dimensional vector space $V$, we write $\GG=\Grass(k,V)$ for the Grassmannian of $k$-dimensional quotients of $V$. We let $\mc{Q}$ denote the tautological quotient sheaf of rank $k$, and let $\mc{R}$ denote the tautological subsheaf of rank $N-k$, which fit into the short exact sequence
$$
0\longrightarrow \mathcal{R} \longrightarrow V\otimes \mathcal{O}_{\GG}\longrightarrow \mathcal{Q}\longrightarrow 0.
$$
We refer to the map $V\otimes \mathcal{O}_{\GG}\to \mathcal{Q}$ as the tautological surjection.

In order to calculate cohomology of locally free sheaves on $\GG$, we use Bott's theorem, which we now recall.

\begin{theorem}\label{Bott_thm}\textnormal{(see \cite[Corollary 4.1.9]{Weyman})}
Using notation as above, let $\lambda \in \mathbb{Z}^k_{\textnormal{dom}}$ and $\mu\in \mathbb{Z}^{N-k}_{\textnormal{dom}}$, and let
	$$
	\gamma=(\lambda_1,\cdots, \lambda_k,\mu_1,\cdots,\mu_{N-k}).
	$$
	Let $\rho = (N-1, N-2, \cdots, 1, 0)$, and denote by $\Sort(\gamma + \rho)$ the element of $\mathbb{Z}^N$ obtained by arranging the entries of $\gamma+\rho$ in weakly decreasing order. Let $\tilde{\gamma} = \Sort(\gamma + \rho) - \rho$. Let $\sigma$ be the minimal number of simple transpositions that permutes $\gamma+\rho$ into $\Sort(\gamma+\rho)$. Then we have
\[
		H^j(\GG, \Schur_\lambda \mc{Q} \otimes \Schur_\mu \mc{R}) = 
	\begin{cases}
		\Schur_{\tilde{\gamma}} V & \text{if }\gamma+\rho \text{ has no repeated entries and } j=\sigma, \\
		0 & \text{otherwise.}
	\end{cases}
	\]
\end{theorem}

We store the following consequences of Theorem \ref{Bott_thm} for use in our argument.

\begin{lemma}\label{Bottlemma}
Let $\lambda \in \mathbb{Z}^k_{\textnormal{dom}}$. The following is true about the cohomology of $\mathbb{S}_{\lambda}\mc{Q}$.
\begin{enumerate}
    \item If $\lambda_p\geq p-k$ and $\lambda_{p+1}\leq p-N$ for some $0\leq p\leq k$, then
    $$
    H^{(N-k)(k-p)}(\mathbb{G},\mathbb{S}_{\lambda}\mc{Q})=\mathbb{S}_{\lambda(p)}V,
    $$
    where 
    $$
    \lambda(p)=(\lambda_1,\cdots, \lambda_p, (p-k)^{N-k},\lambda_{p+1}+(N-k),\cdots, \lambda_k+(N-k)).
    $$
    \item Otherwise, we have 
$$
H^{j}(\mathbb{G},\mathbb{S}_{\lambda}\mc{Q})=0\quad \textnormal{for all $j$}.
$$
\end{enumerate}

\end{lemma}

\begin{proof}
See \cite[Theorem 2.3(b)]{RaicuWeyman}.
\end{proof}

We also have the following vanishing result.

\begin{lemma}\label{Bottlemma2}
Let $\lambda \in \mathbb{Z}^k_{\textnormal{dom}}$. If $\lambda_p\geq p-k$ for some $0\leq p\leq k$, then 
$$
H^{j}(\mathbb{G},\mathbb{S}_{\lambda}\mc{Q}\otimes \mathcal{R})=0\quad \textnormal{for all $j\geq (N-k)(k-p)+1$}.
$$
\end{lemma}

\begin{proof}
We apply Bott's algorithm to $\gamma = (\lambda_1,\cdots,\lambda_k, 1,0,\cdots,0)$. Let $\rho=(N-1,N-2,\cdots,1,0)$, so that 
	$$
	\gamma+\rho=(\lambda_1+N-1,\cdots,\lambda_k+N-k, N-k,N-k-2,\cdots,0).
	$$
Since $\lambda_p\geq p-k$ we have $(\gamma+\rho)_p\geq N-k$. On the other hand, we have $(\gamma+\rho)_{k+1}=N-k$. Thus, when sorting $\gamma+\rho$ into strictly decreasing order, $(\gamma+\rho)_{k+1}$ is not moved past $(\gamma+\rho)_p$. In particular, there are at most $(N-k)(k-p)$ simple transpositions when applying Bott's algorithm to sort $(\gamma+\rho)$. By Theorem \ref{Bott_thm} we obtain the desired result.
\end{proof}

\subsection{Geometric vector bundles and spaces of relative matrices}\label{relativeSetting}

Given a locally free sheaf $\mathcal{E}$ of finite rank on a variety $B$, we write $\Sym_B(\mathcal{E})$ for the symmetric algebra on $\mathcal{E}$ over $B$. The \defi{geometric vector bundle} associated to $\mathcal{E}$ is (see \cite[Exercise II.5.18]{hartshorne}):
$$
\mathbb{A}_B(\mathcal{E}):=\mathbf{Spec}\left(\Sym_B(\mathcal{E})\right).
$$
We will be interested in geometric vector bundles of the form
$$
X_B(\mathcal{E}_1,\mathcal{E}_2):=\mathbb{A}_B(\mathcal{E}_1\otimes \mathcal{E}_2),
$$
where $\mathcal{E}_1$ and $\mathcal{E}_1$ are locally free sheaves of ranks $n_1$ and $n_2$, respectively. We have that $X_B(\mathcal{E}_1,\mathcal{E}_2)$ is locally isomorphic to $n_1\times n_2$ matrices over $B$, and thus admits a stratification by matrix rank. We write $O_{B,p}(\mathcal{E}_1,\mathcal{E}_2)$ for the strata of matrices of rank $p$, for $0\leq p\leq \min(n_1,n_2)$.

Similarly, for a locally free sheaf $\mathcal{E}$ of rank $n$, we write $X^{skew}_B(\mathcal{E}):=\mathbb{A}_B(\bigwedge^2\mathcal{E})$, which is locally isomorphic to $n\times n$ skew-symmetric matrices over $B$. We write $O_{B,p}(\mathcal{E})$ for the strata of skew-symmetric matrices of rank $2p$, for $0\leq p\leq \lfloor n/2 \rfloor$.

\subsection{A consequence of graded local duality}\label{localDualSection}
In this subsection, we let $S=\mathbb{C}[x_1,\cdots,x_d]$, endowed with the standard grading and homogeneous maximal ideal $\mathfrak{m}$. We let $(-)^\vee$ denote the graded Matlis dual functor \cite[Section 3.6]{BrunsHerzog}, with applied to a graded $S$-module $M$ satisfies
\begin{equation*}
    (M^\vee)_\ell = \Hom_\CC(M_{-\ell}, \CC),\quad \textnormal{for $\ell\in \mathbb{Z}$}. 
\end{equation*}

We store the following consequence of graded local duality for use in our arguments. We define the minimal socle degree of $M$ to be the minimal degree of a minimal generator of $\Soc(M)$.

\begin{lemma}\label{minimalsocle}
Let $M$ be a finitely generated graded $S$-module. Then we have 
$$\Soc(H^{d-j}_\frakm(M))^\vee \cong  \Ext^j_S(M, S(-d))\otimes_S S/\mathfrak{m}.$$

In particular, the negative of the minimal socle degree of $H^{d-j}_\frakm(M)$ is the same as the maximal degree of a minimal generator of $\Ext^j_S(M, S(-d))$.
\end{lemma}
\begin{proof}
By \cite[Exercise 12]{Huneke}, we have the following isomorphism for a finitely-generated $S$-module $A$, and any $S$-module $B$:
\begin{equation*}
    \Ext^i_S(A, B)^\vee  \cong \Tor^S_i(A,B^\vee), 
\end{equation*}
a fact which can be proven via Hom-tensor adjointness, in a similar manner to the argument in \cite[Example 3.6]{Huneke} as follows: Take a free resolution $F^\bullet$ of $A$, then the homology of $F^\bullet \otimes_S B^\vee$ is $\Tor^S_i(A,B^\vee)$. Moreover, since $(F^\bullet \otimes_S B^\vee)^\vee \cong \Hom_S(F^\bullet, B)$, taking the cohomology and applying the Matlis dual on both sides result in the desired isomorphism.

Setting $i=0$, $A=S/\mathfrak{m}$, and $B=H^{d-j}_\frakm(M)$, we obtain
\begin{equation}\label{consOfHuneke}
    \Soc(H^{d-j}_\frakm(M))^\vee \cong \Hom_S(S/\mathfrak{m}, H^{d-j}_{\frakm}(M))^\vee 
    \cong \Tor^S_0(S/\mathfrak{m}, H^{d-j}_\frakm(M)^\vee).
\end{equation}
By graded local duality \cite[Theorem 3.6.19]{BrunsHerzog}, we have in the category of graded modules,
\begin{equation}\label{gld}
    H^{d-j}_\frakm(M)^\vee \cong \Ext^j_S(M, S(-d)).
\end{equation}
Combining (\ref{consOfHuneke}) and (\ref{gld}) gives the desired result.
\end{proof}

\section{Generic matrices}

In this section, we consider the polynomial ring $S=\Sym(\CC^m \otimes \CC^n)$ on the space of $m\times n$ generic matrices, with $m\geq n$. This ring is endowed with an action of the group $\GL=\GL_m(\CC)\times \GL_n(\CC)$.

\subsection{$\GL$-equivariant $\D$-modules on general matrices}\label{sectionDgen} Let $\mathcal{D}$ denote the Weyl algebra of differential operators on $S$. All of our $\D$-modules are left $\D$-modules. For each $p=0,\cdots ,n$, we write $D_p$ for the intersection homology $\D$-module associated to the trivial local system on matrices of rank $p$ \cite[Definition 3.4.1]{hotta2007d}. As we are only interested in the $S$-module structure of $D_p$, the reader less versed in the theory of $\D$-modules may take the descriptions (\ref{Raicu18quot}) and (\ref{lcmaxminors}) below as definitions of the modules $D_p$ for the purposes of this paper. 

The module $D_p$ decomposes into irreducible $\GL$-representations as follows \cite[Section 5]{Raicu16}:
\begin{equation}\label{Dp_representation}
		D_p = \bigoplus_{\lambda\in W^p} \Schur_{\lambda(p)} \CC^m \otimes \Schur_\lambda \CC^n,
\end{equation}
	where $\lambda(p)=(\lambda_1,\cdots ,\lambda_p,(p-n)^{m-n},\lambda_{p+1}+(m-n),\cdots ,\lambda_n+(m-n))$, and
	$$
	W^p = \{\lambda \in \ZZ_{\Dom}^n : \lambda_p \geq p-n,\; \lambda_{p+1}\leq p-m \}.
	$$
For example, we have that $D_n=S$, and (\ref{Dp_representation}) recovers the \defi{Cauchy identity}
\begin{align}\label{cauchy}
	S = \bigoplus_{\lambda=(\lambda_1\geq \cdots \geq \lambda_n\geq 0)} \Schur_\lambda \CC^m \otimes \Schur_\lambda \CC^n.
\end{align}
In the case of square matrices ($m=n$), each $D_p$ appears as a $\D$-simple composition factor of $S$ localized at the $n\times n$-determinant $\Det := \Det(x_{ij})$, where $\D$ acts on $S_{\det}$ via natural differentiation (by the quotient rule). Indeed, if we write $\langle \Det^{-i} \rangle_{\D}$ for the $\DD$-submodule of $S_{\Det}$ generated by $\Det^{-i}$, then the composition series of $S_{\det}$ is given by \cite[Theorem 1.1]{Raicu16}
	\begin{align}\label{filtration_set}
		0 \subsetneq S  \subsetneq \langle \Det^{-1} \rangle_\DD \subsetneq \langle \Det^{-2} \rangle_\DD \subsetneq ... \subsetneq \langle \Det^{-n} \rangle_\DD  = S_{\Det},
	\end{align}
	from which we get the simple $\DD$-module 
	\begin{equation}\label{Raicu18quot}
	 D_p = \langle \Det^{p-n} \rangle_\DD / \langle \Det^{p-n+1} \rangle_\DD.   
	\end{equation}
	
In the case of non-square matrices, each module $D_p$ arises as a local cohomology module of $S$ with support in the ideal of maximal minors $I=I_n$. Indeed, the nonzero local cohomology modules are given by \cite[(5.1)]{Raicu16} (see also \cite{RaicuWeymanWitt, RaicuWeyman}):
\begin{equation}\label{lcmaxminors}
H_I^{(m-n)(n-p)+1}(S)=D_p,\quad \textnormal{for $p=0,\cdots,n-1$.}    
\end{equation}

\subsection{The $S$-module structure of $D_p$}\label{genericModuleSection}

We freely use notation from Section \ref{sectionDgen}. Given $\lambda\in W^p$, and a $\GL$-equivariant $S$-submodule $M\subseteq D_p$, we write $M_{\lambda}$ for the $\GL$-isotypic component of $M$ corresponding to $\Schur_{\lambda(p)} \CC^m \otimes \Schur_\lambda \CC^n$ (by equivariant $S$-submodule, we mean an $S$-submodule that is preserved by the $\GL$ action on $D_p$).

The main result of this subsection is the following.

\begin{theorem}\label{S_structure_Dp}
	Let $0\leq p\leq n$ and let $\lambda,\mu \in W^p$. 
	\begin{enumerate}
		\item We have that $D_{p,\lambda}$ generates $D_{p,\mu}$ over $S$ if and only if $\lambda \leq \mu$.
		\item In particular, if $M$ is a $\GL$-equivariant $S$-submodule of $D_p$, then $M$ is generated by $M_{\lambda}$ over all $\lambda$ minimal for which $M_{\lambda}\neq 0.$
	\end{enumerate} 
\end{theorem}

Since $D_p$ is multiplicity-free as a representation of $\GL$ (meaning each irreducible representation appears at most once), the second assertion follows from the first.	

We note first that Theorem \ref{S_structure_Dp} holds in the case $p=n$.

\begin{theorem}\textnormal{(see \cite[Theorem 4.1, Corollary 4.2]{deConciniEisenbudProcesi})}\label{dCEP_module_structure}
Given two partitions $\lambda$ and $\mu$, we have that $S_{\lambda}$ generates $S_{\mu}$ if and only if $\lambda \leq \mu$.
\end{theorem}

Next, we address the case of square matrices.

\begin{lemma}\label{SstructureSquare}
If $m=n$ then Theorem \ref{S_structure_Dp} holds.
\end{lemma}

\begin{proof}
We use the filtration (\ref{filtration_set}), and the fact that each $D_p$ appears as a subquotient. Since $S_{\Det}=\bigcup_{d\geq 0} S\cdot \Det^{-d}$, it follows from (\ref{cauchy}) we have
\begin{equation}
	S_{\Det} = \bigoplus_{\lambda \in \ZZ^n_{\Dom}} \Schur_\lambda \CC^n \otimes \Schur_\lambda \CC^n.
\end{equation}
Another way to see this is that $S_{\Det}$ is a direct sum of $D_0,\cdots, D_n$ as a representation of $\GL$. Suppose that $\lambda, \mu \in \ZZ^n_{\Dom}$ and $\lambda \leq \mu$. Let $a=\lambda_n$. If $a\geq 0$ then $S_{\Det,\lambda}$ and $S_{\Det,\mu}$ belong to $S$, so we get that $S_{\Det,\lambda}$ generates $S_{\Det,\mu}$ by Theorem \ref{dCEP_module_structure}. Otherwise, twist by $\det^{-a}$ and use Theorem \ref{dCEP_module_structure} to get that $S_{\Det,\lambda}\cdot \det^{-a}$ generates $S_{\Det,\mu}\cdot \det^{-a}$. Therefore, $S_{\Det,\lambda}$ generates $S_{\Det,\mu}$.

	Since $D_p$ is the quotient $\langle \Det^{p-n} \rangle_\DD / \langle \Det^{p+1-n} \rangle_\DD$, we only need to show that if $\mu$ and $\lambda$ are in $W^p$ and if $\mu \geq \gamma \geq \lambda$ then $\gamma \in W^p$ as well. Since $\lambda \in W^p$, we have $\gamma_p \geq \lambda_p \geq p-n$, and since $\mu \in W^p$, we have $p-n \geq \mu_{p+1} \geq \gamma_{p+1}$. Therefore $\gamma \in W^p$.
\end{proof}

Now assume that $m>n$. We consider the Grassmannian $\mathbb{G}=\operatorname{Gr}(n,\mathbb{C}^m)$ of $n$-dimensional quotients of $\mathbb{C}^m$, with tautological rank $n$ sheaf $\mathcal{Q}$. We will freely use terminology from Section \ref{relativeSetting} on geometric vector bundles and spaces of relative matrices.

We write $Y=X_{\mathbb{G}}(\mathcal{Q},\mathbb{C}^n)$ for the geometric vector bundle associated to the locally free sheaf $\mathcal{Q}\otimes \mathbb{C}^n$, and for $0\leq p\leq n$ we write $O_p^Y=O_{\mathbb{G},p}(\mathcal{Q},\mathbb{C}^n)$ for the strata of matrices of rank $p$. We have the following diagram:
$$
\begin{tikzcd}
Y \arrow[r, hook, "s"] \arrow[dr, "f"] & \mathbb{G} \times \mathbb{C}^{m\times n}  \arrow[d, "\pi"]\\
& \mathbb{C}^{m\times n}
\end{tikzcd}
$$
where $s$ is the inclusion, $\pi$ is the projection, and $f=\pi\circ s$. For $0\leq p\leq n$, we let $O_p$ denote the strata in $\mathbb{C}^{m\times n}$ of matrices of rank $p$. We note that $f$ is a birational isomorphism on the set of matrices of rank $n$, so that $f^{-1}(O_n)=O_n^Y=O_n$. 

Let $S=\Sym(\CC^m\otimes \CC^n)$ be the ring of polynomial functions on $\mathbb{C}^{m\times n}$. We abuse notation and identify $S=\mc{O}_{\mathbb{G}}\otimes S$, viewing $S$ as a sheaf of graded algebras on $\mathbb{G}$. Further, we identify the category of quasi-coherent graded $S$-modules on $\mathbb{G}$ with the category of quasi-coherent sheaves on $\mathbb{C}^{m\times n}\times \mathbb{G}$, graded with respect to the affine space $\mathbb{C}^{m\times n}$ (see \cite[Exercise II.5.17]{hartshorne}).

For $0\leq p\leq n$ we write $\mathcal{L}(O_p^Y, Y)$ for the intersection homology $\D_Y$-module associated to the trivial local system on $O_p^Y$, and we define $D_p^{\mathcal{Q}}=s_{\ast}\mathcal{L}(O_p^Y,Y)$, endowed with its natural structure as a $S^{\mc{Q}}=s_{\ast}\mc{O}_Y=\operatorname{Sym}(\mc{Q}\otimes \mathbb{C}^n)$-module (we write $s_{\ast}$ for the direct image of quasi-coherent sheaves). By (\ref{Dp_representation}) on the space $X_{\mathbb{G}}(\mathcal{Q},\mathbb{C}^n)$ we have the following decomposition of $\mathcal{O}_{\mathbb{G}}$-modules
\begin{equation}\label{relDp}
D_p^{\mathcal{Q}}\cong\bigoplus_{\lambda\in W^p} \mathbb{S}_{\lambda}\mc{Q}\otimes \mathbb{S}_{\lambda}\mathbb{C}^n.
\end{equation}
Given $\lambda \in W^p$, we write $D_{p,\lambda}^{\mathcal{Q}}=\mathbb{S}_{\lambda}\mc{Q}\otimes \mathbb{S}_{\lambda}\mathbb{C}^n$ for the $\lambda$-th isotypic component of $D_p^{\mathcal{Q}}$. Lemma \ref{SstructureSquare} implies that 
\begin{equation}\label{consOfSquare}
\textnormal{$D^{\mc{Q}}_{p,\lambda}$ generates $D^{\mc{Q}}_{p,\mu}$ over $S^{\mc{Q}}$}\quad \iff \quad \lambda \leq \mu.
\end{equation}
Indeed, the direction ``$\implies$" of (\ref{consOfSquare}) follows from the Littlewood-Richardson Rule \cite[Theorem 2.3.4]{Weyman} applied to $S^{\mc{Q}}_{(1)}\otimes D^{\mathcal{Q}}_{p,\lambda}$, whereas the  ``$\impliedby$" direction follows from Lemma \ref{SstructureSquare} since $S^{\mathcal{Q}}$ locally (over an affine open subset of $\mathbb{G}$) looks like the polynomial ring $\Sym(\CC^n\otimes \CC^n)$, and $D_p^{\mc{Q}}$ locally looks like module $D_p$ on the space of $n\times n$ matrices.

The following result relates $D_p^{\mc{Q}}$ via $\pi$ to the $S$-module $D_p$.

\begin{lemma}\label{pushdownGen}
Let $0\leq p\leq n$. We have the following.
\begin{enumerate}
\item $\mathbb{R}^j\pi_{\ast}(D_p^{\mc{Q}})\neq 0$ if and only if $j=(m-n)(n-p)$,
\item there is an isomorphism of $S$-modules $\mathbb{R}^{(m-n)(n-p)}\pi_{\ast}(D_p^{\mc{Q}})\cong D_p$.
\end{enumerate}
\end{lemma}

\begin{proof}
Since $\mathbb{C}^{m\times n}$ is affine, we have $\mathbb{R}^j\pi_{\ast}D_p^{\mc{Q}}\cong H^j(\mathbb{G} \times \mathbb{C}^{m\times n},D_p^{\mc{Q}})$, and since the projection from $\mathbb{G} \times \mathbb{C}^{m\times n}$ to $\mathbb{G}$ is an affine morphism, we conclude that 
$$
\mathbb{R}^j\pi_{\ast}D_p^{\mc{Q}}\cong \bigoplus_{\lambda\in W^p}H^j(\mathbb{G},\mathbb{S}_{\lambda}\mc{Q}\otimes \mathbb{S}_{\lambda}\mathbb{C}^n).
$$
We note that $H^j(\mathbb{G},\mathbb{S}_{\lambda}\mc{Q}\otimes \mathbb{S}_{\lambda}\mathbb{C}^n) = H^j(\mathbb{G}, \Schur_\lambda \mc{Q}) \otimes \Schur_\lambda \CC^n$. Then by (\ref{relDp}) and Lemma \ref{Bottlemma} we obtain the first assertion, and we obtain that the claimed isomorphism in the second assertion holds in the category of $\GL$-representations. The isomorphism holds in the category of $S$-modules because these direct images are calculating local cohomology with support in maximal minors. Indeed, let $i$ denote the inclusion $i: O_n \hookrightarrow \CC^{m\times n}$, and let $i'$ denote the inclusion of $O^Y_n=O_n$ into $Y$. We define
$$
S^{\mathcal{Q}}_{\det}=s_{\ast}i'_{\ast}\mathcal{O}_{O_n^Y},
$$
which is the ``relative version" of the localization at the $n\times n$ determinant. Since $\pi\circ s\circ i'=i$ and $i'$ is affine, we have the following isomorphisms of $S$-modules for all $j\geq 1$:
\begin{equation}\label{lcdescription}
\mathbb{R}^j\pi_{\ast}(S^{\mathcal{Q}}_{\det})\cong\mathbb{R}^ji_{\ast}(\mathcal{O}_{O_n})\cong H^{j+1}_I(S),
\end{equation}
where the latter isomorphism follows from \cite[Exercise III.2.3(e)]{hartshorne} for $\mathscr{F}=\mc{O}_{\CC^{m\times n}}$ and $U=O_n$. Since $S^{\mathcal{Q}}_{\det}$ has composition factors $D_0^{\mathcal{Q}},\cdots , D_n^{\mathcal{Q}}$, each with multiplicity one, and $\mathbb{R}^{j}\pi_{\ast}(D_p^{\mc{Q}})$ is nonzero if and only if $j=(m-n)(n-p)$, we have that $\mathbb{R}^{(m-n)(n-p)}\pi_{\ast}(S^{\mathcal{Q}}_{\det})\cong \mathbb{R}^{(m-n)(n-p)}\pi_{\ast}(D_p^{\mc{Q}})$. Thus, by (\ref{lcdescription}) and (\ref{lcmaxminors}), we have $\mathbb{R}^{(m-n)(n-p)}\pi_{\ast}(D_p^{\mc{Q}})\cong D_p$.  
\end{proof}

We use Lemma \ref{pushdownGen} and (\ref{consOfSquare}) to prove Theorem \ref{S_structure_Dp} in the case $m>n$.

\begin{proof}[Conclusion of proof of Theorem \ref{S_structure_Dp}]
Let $\lambda, \mu \in W^p$ satisfy $\mu\geq \lambda$ and $|\mu|=|\lambda|+1$. By (\ref{consOfSquare}), via the surjection $S\twoheadrightarrow S^{\mc{Q}}$ the multiplication map $S\otimes D_p^{\mc{Q}}\to D_p^{\mc{Q}}$ induces a surjective morphism
\begin{equation}\label{multDpqgen}
(\CC^m\otimes \mathbb{C}^n)\otimes D^{\mc{Q}}_{p,\lambda}\twoheadrightarrow D^{\mc{Q}}_{p,\mu}.
\end{equation}
Using Lemma \ref{pushdownGen}, we apply $\mathbb{R}^{(m-n)(n-p)}\pi_{\ast}(-)$ to this surjection to obtain a morphism
\begin{equation}\label{desiredMap}
S_1\otimes D_{p,\lambda}\longrightarrow D_{p,\mu},   
\end{equation}
which is induced from the multiplication map $S\otimes D_p \to D_p$ (here $S_1=\CC^m\otimes \mathbb{C}^n$ denotes the first graded piece of $S$). 

To complete the proof, we need to show that (\ref{desiredMap}) is surjective (or equivalently, it is nonzero). To this end, let $\mc{K}_{\lambda,\mu}$ denote the kernel of (\ref{multDpqgen}). It suffices to show that $\mathbb{R}^{(m-n)(n-p)+1}\pi_{\ast}(\mc{K}_{\lambda,\mu})=0$. By construction, the map (\ref{multDpqgen}) factors as follows
$$
(\CC^m\otimes \mathbb{C}^n)\otimes D^{\mc{Q}}_{p,\lambda}\twoheadrightarrow (\mc{Q}\otimes \mathbb{C}^n)\otimes D^{\mc{Q}}_{p,\lambda} \twoheadrightarrow D^{\mc{Q}}_{p,\mu},
$$
where the left map is induced from the tautological surjection $S\twoheadrightarrow S^{\mc{Q}}$, and thus has kernel $(\mc{R}\otimes \mathbb{C}^n)\otimes D^{\mc{Q}}_{p,\lambda}$. The right map is induced from the multiplication map of $S^{\mc{Q}}$ on $D_p^{\mc{Q}}$, so by Pieri's Rule \cite[Collorary 2.3.5]{Weyman} it has kernel with summands of the form $\mathbb{S}_{\sigma}\mc{Q}\otimes \mathbb{S}_{\nu}\mathbb{C}^n$, where $\sigma$ and $\nu$ satisfy $\sigma \geq \lambda$, $\nu\geq \lambda$, and $|\sigma|=|\nu|=|\lambda|+1$.

Thus, $\mc{K}_{\lambda,\mu}$ admits a filtration with composition factors of the form
\begin{itemize}
    \item $(\mc{R}\otimes \mathbb{C}^n)\otimes D^{\mc{Q}}_{p,\lambda}$,
    \item $\mathbb{S}_{\sigma}\mc{Q}\otimes \mathbb{S}_{\nu}\mathbb{C}^n$, for $\sigma\geq \lambda$ and $\nu \geq \lambda$ with $|\sigma|=|\nu|=|\lambda|+1$.
\end{itemize}
By Lemma \ref{Bottlemma} and Lemma \ref{Bottlemma2} for $N=m$ and $k=n$, these bundles do not have cohomology in degree $(m-n)(n-p)+1$. Therefore, $\mathbb{R}^{(m-n)(n-p)+1}\pi_{\ast}(\mc{K}_{\lambda,\mu})=0$, as required.
\end{proof}

\subsection{Calculation of degree of socle generators}\label{cal_deg_socle_generic}
In this subsection, we determine the $S$-module structure of $\Ext^j_S(S/I^t, S)$, and as a consequence obtain Theorem \ref{main}(1) and deduce Corollary \ref{main_corollary}(1). By the proof of \cite[Theorem 4.5]{RaicuWeymanWitt} (see also  \cite[Main Theorem]{Raicu18}), for all $t>0$ the natural inclusion $I^{t+1}\subseteq I^t$ induces an injective map
$$\Ext^j_S(S/I^t, S) \hookrightarrow \Ext^j_S(S/I^{t+1},S).$$ 
Since local cohomology is described by the direct limit
$$
H^j_I(S) = \varinjlim_t \Ext^j_S(S/I^t, S), 
$$
we have that $\Ext^j_S(S/I^t)$ is an $S$-submodule of $H_I^j(S)$. We recall the structure of $\Ext^j_S(S/I^t, S)$ as $\GL$-representations from \cite[Theorem 4.3]{RaicuWeymanWitt}: for $0\leq p\leq n-1$ we have the multiplicity-free decomposition

\begin{equation}\label{ext_GL_decomposition_generic}
\operatorname{Ext}^{(m-n)(n-p)+1}_S(S/I^t,S)=\bigoplus_{\substack{\lambda \in W^p \\ \lambda_n \geq -t- (m-n)}} \mathbb{S}_{\lambda(p)}\mathbb{C}^m \otimes \mathbb{S}_{\lambda}\mathbb{C}^n,
\end{equation}
which is a $\GL$-equivariant $S$-submodule of $D_p$.

We store the following lemma for use in our argument below.

\begin{lemma}\label{onedimlemma}
Let $0\leq p\leq n<m$ and let $\lambda \in W^p$. Then $\Schur_{\lambda(p)}\CC^m\otimes \Schur_{\lambda}\CC^n$ is one-dimensional if and only if $p=0$ and $\lambda=((-m)^n)$. 
\end{lemma}

\begin{proof}
If $\lambda \in W^p$ and $p>0$, then $\lambda_p>\lambda_{p+1}$, so $\Schur_{\lambda}\CC^n$ is not one-dimensional. When $p=0$, we have that $\Schur_{\lambda}\CC^n$ is one-dimensional if and only if $\lambda=((-m)^n)$, in which case $\lambda(0)=((-n)^m)$, so that $\Schur_{\lambda(0)}\CC^m$ is one-dimensional.
\end{proof}

 Using Theorem \ref{S_structure_Dp} and the discussion above, we get

\begin{theorem}\label{degrees_gens_ext_generic}
    If the module $\Ext^{(m-n)(n-p)+1}_S(S/I^t, S)$ is nonzero, then it is generated by the irreducible subrepresentation $\Schur_{\lambda(p)}\CC^m\otimes \Schur_{\lambda}\CC^n$, where 
$$
\lambda_1=\cdots = \lambda_p = p-n,\quad \lambda_{p+1}=\cdots =\lambda_n=-t-(m-n).
$$    
 In particular  $\Ext^{(m-n)(n-p)+1}_S(S/I^t, S)$ is generated in a single degree, which is $(p-n)(t+p+m-n)$. Moreover, this module is cyclic if and only if $p=0$ and $t=n$. 
\end{theorem}

\begin{proof}
By (\ref{ext_GL_decomposition_generic}), $\lambda$ is minimal with respect to  $\geq$ (see Section \ref{prelimRep}) for which $\Schur_{\lambda(p)}\CC^m\otimes \Schur_{\lambda}\CC^n$ is a subrepresentation $\Ext^{(m-n)(n-p)+1}_S(S/I^t, S)$. 
Thus, by Theorem \ref{S_structure_Dp} we conclude that the module $\Ext^{(m-n)(n-p)+1}_S(S/I^t, S)$ is generated by $\Schur_{\lambda(p)}\CC^m\otimes \Schur_{\lambda}\CC^n$, which has degree
$$
p(p-n)+(n-p)(-t-(m-n))=p (p-n)+(p-n) (t+(m-n))= (p-n)(t+p+(m-n)).
$$
To prove the assertion about cyclicity, note that by Lemma \ref{onedimlemma}, $\Ext^{(m-n)(n-p)+1}_S(S/I^t, S)$ is cyclic if and only if $p=0$ and $\lambda=((-m)^n)$ is the minimal weight. By (\ref{ext_GL_decomposition_generic}) we conclude that $t=n$. 
\end{proof}

Finally, by Lemma \ref{minimalsocle} and Theorem \ref{degrees_gens_ext_generic}, we determine the degrees of the socle generators of local cohomology.
\begin{corollary}
    Let $0\leq p\leq n-1$. The module $\Soc(H^{p(m-n)+n^2-1}_\frakm(S/I^t))$ is generated in the single degree $(n-p)(t+p+m-n) - mn$. 
\end{corollary}

We carry out another example.

\begin{example}\label{Ex1}
Let $m=4$, $n=2$, and $p=1$, in which case we are considering cohomological degree $3$ of $\Ext$. By \cite[Theorem 4.3]{RaicuWeymanWitt}, we have that $$\Ext^3_S(S/I^t,S)\neq 0\quad \iff \quad t\geq 1,$$
and there is an infinite chain of inclusions
$$\Ext^3_S(S/I, S)\subseteq \Ext^3_S(S/I^2, S)\subseteq \Ext^3_S(S/I^3, S)\subseteq \cdots,$$
and the module $\Ext^3_S(S/I^4, S)$ has the following decomposition into irreducible representations:
\begin{align*}
\Ext^3_S(S/I^4, S) & = \left(\bigoplus_{k\geq -1} \Schur_{(k,-1^2,-4)}\CC^4 \otimes \Schur_{(k,-6)}\CC^2\right)\oplus\left(\bigoplus_{k\geq -1} \Schur_{(k,-1^2,-3)}\CC^4 \otimes \Schur_{(k,-5)}\CC^2\right)\\
& \oplus \left(\bigoplus_{k\geq -1} \Schur_{(k,-1^2,-2)}\CC^4 \otimes \Schur_{(k,-4)}\CC^2\right)\oplus\left( \bigoplus_{k\geq -1} \Schur_{(k,-1^3)}\CC^4 \otimes \Schur_{(k,-3)}\CC^2\right).
\end{align*}
By Theorem \ref{S_structure_Dp}, $\Schur_{(-1^3, -4)}\CC^4 \otimes \Schur_{(-1, -6)}\CC^2$ is the space of generators of $\Ext^3_S(S/I^4, S)$ with degree $-7$. The second, third and fourth direct summands are the decomposition of $\Ext^3_S(S/I^3, S)$, which is generated by $\Schur_{(-1^3, -3)}\CC^4 \otimes \Schur_{(-1, -5)}\CC^2$. The third and fourth direct summands are the decomposition of $\Ext^3_S(S/I^2, S)$, which is generated by $ \Schur_{(-1^3, -2)}\CC^4 \otimes \Schur_{(-1,-4)}\CC^2$. Finally, the fourth direct summand is the decomposition of  $\Ext^3_S(S/I, S)$, which is generated by $\Schur_{(1,-1^3)}\CC^4 \otimes \Schur_{(1,-3)}\CC^2$.
\end{example}

\begin{remark}
One could attempt the above techniques to study socle degrees of $H^j_\frakm(S/I^t)$ or $H^j_\frakm(S/I^{(t)})$, where $I=I_p$ is the determinantal ideal of $p\times p$-minors \textnormal{($1\leq p\leq n$)}, and  $I^{(t)}$ is the $t$-th symbolic power of $I$. However, when $p<n$ the natural map $\Ext^j_S(S/I^t, S)\to H^j_I(S)$ is not generally an inclusion, so it is difficult to deduce information about generators of $\Ext^j_S(S/I^t, S)$ from $H^j_I(S)$. On the other hand, the maps $\Ext^j_S(S/I^{(t)}, S)\to H^j_I(S)$ are always inclusions \cite[Corollary 5.9]{Raicu18}. However, the local cohomology module $H^j_I(S)$ is not in general a multiplicity-free representation, which makes it more difficult to distinguish between irreducible representations that are generators of $\Ext^j_S(S/I^{(t)}, S)$, and those that are generated by others.
\end{remark}

\section{Skew-symmetric matrices}
We let $S = \Sym(\bigwedge^2 \CC^m)$ be the polynomial ring on the space of $m\times m$ skew-symmetric matrices. This ring is endowed with a group action of $\GL=\GL_m(\CC)$. We write $n = \lfloor m/2 \rfloor$, so that $m=2n$ if $m$ is even, and $m=2n+1$ if $m$ is odd.

\subsection{Equivariant $\D$-modules on skew-symmetric matrices}\label{DmodSkew} Let $\mathcal{D}$ denote the Weyl algebra of differential operators on $S$, and we continue to work in the category of left $\D$-modules. For each $p = 0,\cdots, n$, we write $B_p$ for the intersection homology $\D$-module associated to the trivial local system on matrices of rank $2p$ \cite[Definition 3.4.1]{hotta2007d}. Again, we are only interested in the structure of $B_p$ as a $\GL$-equivariant $S$-module, so the reader may take the descriptions below as the definition of $B_p$ for the purposes of this paper. 

The module $B_p$ decomposes into irreducible $\GL$-representations as follows \cite[Section 6]{Raicu16}:
\begin{equation}\label{Bp_representation}
		B_p = \bigoplus_{\lambda \in U^p} \mathbb{S}_\lambda \mathbb{C}^m,
\end{equation}
where if $m=2n$ is even, then
\begin{equation*}
	U^p = \{\lambda \in \ZZ^{m}_{\text{dom}}: \lambda_{2p} \geq 2(p-n),\; \lambda_{2p+1} \leq 2(p-n)+1, \lambda_{2i-1} = \lambda_{2i} \text{ for all $i$ }\},\end{equation*}
and if $m=2n+1$ is odd, then 
\begin{equation*}
	U^p = \{\lambda \in \ZZ^{m}_{\text{dom}}: \lambda_{2p+1} = 2(p-n), \lambda_{2i} = \lambda_{2i-1} \text{ for } i \leq p,\; \lambda_{2i} = \lambda_{2i+1} \text{ for } p < i \leq n\}.
\end{equation*}
In particular, $S=B_n$ and by \cite[Proposition 2.3.8]{Weyman} we have 
\begin{equation}\label{cauchy_skew}
	S = \bigoplus_{ \lambda_n\geq 0} \Schur_{(\lambda_1,\lambda_1, \lambda_2, \lambda_2,\cdots, \lambda_n,\lambda_n)} \CC^m.
\end{equation}
In other words, $S$ is the multiplicity-free representation whose isotypic components are indexed by Young diagrams with columns of even length. 

When $m=2n$ is even, we denote by $\Pf$ the $m \times m$ Pfaffian. By \cite[Section 6]{Raicu16} we have the following composition series of $S_{\Pf}$: 
\begin{equation}\label{compseriesSkew}
	0 \subsetneq S \subsetneq \langle \Pf^{-2} \rangle_\D \subsetneq \langle \Pf^{-4} \rangle_\D \subsetneq \cdots \subsetneq  \langle \Pf^{-m} \rangle_\D= S_{\Pf},
\end{equation}
where $\langle \Pf^i \rangle_{\D}$ is the $\D$-submodule of $S_{\Pf}$ generated by $i$-th powers of $\Pf$, and from which we get the simple $\D$-module $B_p = \langle \Pf^{-2p}\rangle_{\D} / \langle \Pf^{-2p+2} \rangle_{\D}$ where $0 \leq p \leq n$. 

When $m=2n+1$ is odd, we write $\Pf$ for the ideal of $2n\times 2n$ Pfaffians, known as the ideal of sub-maximal Pfaffians. By \cite[(6.1)]{Raicu16} (see also \cite{RaicuWeymanWitt, RaicuWeyman2}) we have for each $0 \leq p \leq n$,
\begin{equation}
	H^{2(n-p)+1}_{\Pf}(S) = B_p, \quad \textnormal{ for }p=0,\cdots,n.
\end{equation}
\subsection{The $S$-module structure of $B_p$}\label{Bgens}
	We freely use notation from Section \ref{DmodSkew}, and prove the analogue of Theorem \ref{S_structure_Dp} for $B_p$. Given a $\GL$-equivariant $S$-submodule $M$ of $B_p$, we write $M_{\lambda}$ for the isotypic component of $M$ corresponding to $\Schur_{\lambda}\CC^m$.
 
\begin{theorem}\label{S_structure_Bp}
	Let $0\leq p\leq n$ and let $\lambda,\mu \in U^p$. 
	\begin{enumerate}
		\item We have that $B_{p,\lambda}$ generates $B_{p,\mu}$ over $S$ if and only if $\lambda \leq \mu$.
		\item In particular, if $M$ is a $\GL$-equivariant $S$-submodule of $B_p$, then $M$ is generated by $M_{\lambda}$ over all $\lambda$ minimal for which $M_{\lambda}\neq 0.$
	\end{enumerate} 
\end{theorem}
Again, as $B_p$ is a multiplicity-free representation of $\GL$, the second assertion follows from the first. As in the case of generic matrices, we first note that the result holds for the polynomial ring. By \cite[Theorem 3.1]{AbeasisdelFra} we have the following.

\begin{theorem}\label{AdF_module_structure}
	Given $\lambda, \mu\in U^n$, we have $S_\lambda$ generates $S_\mu$ if and only if $\lambda \leq \mu$.
\end{theorem}

Next, we address the case $m$ is even using the existence of the $m\times m$ Pfaffian $\Pf$. 

\begin{lemma}\label{SstructureSquare_skew}
	If $m = 2n$, then Theorem \ref{S_structure_Bp} holds.
\end{lemma}
\begin{proof}
Using (\ref{compseriesSkew}) and Theorem \ref{AdF_module_structure}, the proof in this situation is similar to that of Lemma \ref{SstructureSquare}: since we know the result is true for $S$, we may twist by the $m\times m$ Pfaffian to conclude that it is also true for the localization $S_{\Pf}$, i.e., $S_{\Pf, \lambda}$ generates $S_{\Pf, \mu}$ if and only of $\lambda\leq \mu$. To deduce the result for $B_p$, we only need to show that if $\mu$ and $\lambda$ are in $U_p$ and if $\mu \geq \gamma \geq \lambda$ (with $\gamma_{2i}=\gamma_{2i-1}$ for all $i=1,\cdots, n$) then $\gamma \in U_p$ as well. This follows from the fact that $\gamma_{2p} \geq \lambda_{2p} \geq 2p-1$ and $\gamma_{2p+1} \leq \mu_{2p+1} \leq 2p$.  
\end{proof}

Next we assume that $m=2n+1$ and consider the Grassmannian $\GG = \Grass(2n, \CC^{2n+1})$ of $2n$-dimensional quotients of $\CC^{2n+1}$, with tautological quotient bundle $\Q$ of rank $2n$. 

Similar to Section \ref{genericModuleSection} (see also Section \ref{relativeSetting}), for $0\leq p\leq n$ we consider the relative versions of the modules $B_p$ on $\bigwedge^2 \CC^{2n+1}\times \mathbb{G}$:
\begin{equation}
B_p^{\mathcal{Q}}=\bigoplus_{\lambda \in U^p} \Schur_\lambda \Q,
\end{equation}
each endowed with the natural $S^{\mathcal{Q}}=\Sym(\bigwedge^2 \mathcal{Q})$-module structure. Given $\lambda,\mu \in U^p$, Lemma \ref{SstructureSquare_skew} implies that
\begin{equation}\label{consOfEven}
\textnormal{$B^{\mc{Q}}_{p,\lambda}$ generates $B^{\mc{Q}}_{p,\mu}$ over $S^{\mc{Q}}$}\quad \iff \quad \lambda \leq \mu.
\end{equation}
Let $\pi$ denote the projection map from $ \bigwedge^2 \CC^{2n+1}\times \mathbb{G}$ to $\bigwedge^2 \CC^{2n+1}$. The following result relates $B_p^{\mathcal{Q}}$ via $\pi$ to the $S$-module $B_p$.

\begin{lemma}\label{pushdownGenB}
Let $0\leq p\leq n$. We have the following.
\begin{enumerate}
    \item  $\mathbb{R}^j\pi_{\ast}(B_p^{\mc{Q}})\neq 0$ if and only if $j=2n-2p$,
    \item there is an isomorphism of $S$-modules $\mathbb{R}^{2n-2p}\pi_{\ast}(B_p^{\mc{Q}})\cong B_p$.
\end{enumerate}
\end{lemma}

\begin{proof}
The proof is similar to that of Lemma \ref{pushdownGen}, and is carried out in \cite[Lemma 5.2]{PerlmanLyubeznik}.
\end{proof}

\begin{proof}[Conclusion of proof of Theorem \ref{S_structure_Bp}]
Let $\lambda, \mu \in U^p$ satisfy $\mu\geq \lambda$ and $|\mu|=|\lambda|+2$. By (\ref{consOfEven}), via the surjection $S\twoheadrightarrow S^{\mc{Q}}$ the multiplication map $S\otimes B_p^{\mc{Q}}\to B_p^{\mc{Q}}$ induces a surjective morphism
\begin{equation}\label{multBpqgen2}
(\bigwedge^2\mathbb{C}^{2n+1})\otimes B^{\mc{Q}}_{p,\lambda}\twoheadrightarrow B^{\mc{Q}}_{p,\mu}.
\end{equation}
Using Lemma \ref{pushdownGenB}, we apply $\mathbb{R}^{2n-2p}\pi_{\ast}(-)$ to this surjection to obtain a morphism
\begin{equation}\label{desiredMap2}
S_1 \otimes B_{p,\lambda}\longrightarrow B_{p,\mu},   
\end{equation}
which is induced from the multiplication map $S\otimes B_p \to B_p$ (here $S_1=\bigwedge^2 \CC^m$ denotes the first graded piece of $S$). 

To complete the proof, we need to show that (\ref{desiredMap2}) is surjective (or equivalently, it is nonzero). To this end, let $\mc{K}_{\lambda,\mu}$ denote the kernel of (\ref{multBpqgen2}). It suffices to show that $\mathbb{R}^{2n-2p+1}\pi_{\ast}(\mc{K}_{\lambda,\mu})=0$. By construction, the map (\ref{multBpqgen2}) factors as follows
$$
(\bigwedge^2 \mathbb{C}^{2n+1})\otimes B^{\mc{Q}}_{p,\lambda}\twoheadrightarrow (\bigwedge^2\mc{Q})\otimes B^{\mc{Q}}_{p,\lambda} \twoheadrightarrow B^{\mc{Q}}_{p,\mu},
$$
where the left map is induced from the surjection $S\twoheadrightarrow S^{\mc{Q}}$, and thus has kernel $(\mc{R}\otimes \mathcal{Q})\otimes B^{\mc{Q}}_{p,\lambda}$ (since $\mathcal{R}$ has rank one, see \cite[Exercise II.5.16(d)]{hartshorne}). The right map is induced from the multiplication map of $S^{\mc{Q}}$ on $B_p^{\mc{Q}}$, so by Pieri's Rule \cite[Collorary 2.3.5]{Weyman} it has kernel with summands of the form $\mathbb{S}_{\nu}\mc{Q}$, where $\nu$ is obtained from $\lambda$ by adding two boxes, not in the same row.

Thus, $\mc{K}_{\lambda,\mu}$ admits a filtration with composition factors of the form
\begin{itemize}
    \item $(\mc{R}\otimes \mathcal{Q})\otimes B^{\mc{Q}}_{p,\lambda}$,
    \item $\mathbb{S}_{\nu}\mc{Q}$, for $\nu\geq \lambda$  with $|\nu|=|\lambda|+2$.
\end{itemize}
By Lemma \ref{Bottlemma} and Lemma \ref{Bottlemma2} for $N=m$, $k=2n$, and $p=2p$, these bundles do not have cohomology in degree $2n-2p+1$. Therefore, $\mathbb{R}^{2n-2p+1}\pi_{\ast}(\mc{K}_{\lambda,\mu})=0$, as required. 
\end{proof}

\subsection{Calculation of degree of socle generators}\label{cal_deg_socle_skew}
In this subsection, we let $m=2n+1$, and we write $\Pf\subseteq S$ for the ideal of sub-maximal Pfaffians. We apply our result of the $S$-module structure of $B_p$ to study that of the $\Ext$-modules, just as we did in Section \ref{cal_deg_socle_generic}. 

We recall the (multiplicity-free) $\GL$-decompositions of the $\Ext$-modules in this setting: by \cite[Theorem 5.3]{RaicuWeymanWitt}, for $0 \leq p \leq n-1$ we have

\begin{equation}\label{ext_GL_decomposition_skew}
\operatorname{Ext}^{2(n-p)+1}_S(S/\Pf^t,S)=\bigoplus_{\substack{\lambda \in U^p \\ \lambda_{2n+1} \geq -t- 1}} \mathbb{S}_{\lambda}\mathbb{C}^{2n+1}.
\end{equation}

By \cite[Theorem 5.3, Corollary 5.4]{RaicuWeymanWitt} for all $t>0$ the inclusions $I^{t+1}\subseteq I^t$ induce injections $$\Ext^j_S(S/\Pf^t, S) \hookrightarrow \Ext^j_S(S/\Pf^{t+1},S).$$

Thus, $\Ext^j_S(S/\Pf^t, S)$ is a $\GL$-equivariant $S$-submodule of $H^j_{\Pf}(S)$. We have the following result.

\begin{theorem}\label{degrees_gens_ext_skew}
    If the module $\Ext^{2(n-p)+1}(S/\Pf^t, S)$ is nonzero, then it is generated by the irreducible subrepresentation $\Schur_\lambda \CC^{2n+1}$ where 
    $$\lambda_1=\cdots = \lambda_{2p+1} = 2(p-n),\quad \lambda_{2p+2}=\cdots =\lambda_{2n+1}=-t-1.$$
    In particular, $\Ext^{2(n-p)+1}(S/\Pf^t, S)$ is generated in a single degree, which is $(p-n)(t+2p+2)$. Moreover, the module is cyclic if and only if $t=2(n-p)+1$, which is the first power such that $\Ext^{2(n-p)+1}_S(S/\Pf^t, S)$ is nonzero.
\end{theorem}

\begin{proof}
The proof is similar to Theorem \ref{degrees_gens_ext_generic}: by (\ref{ext_GL_decomposition_skew}) $\lambda$ is minimal with respect to $\geq$, and therefore by Theorem \ref{S_structure_Bp} we have that $\Schur_\lambda \CC^{2n+1}$ generates $\Ext^{2(n-p)+1}(S/\Pf^t, S)$.

Thus, the module is generated in the
single degree
$$\frac{1}{2}((2p+1)(2p-2n) + (-t-1)(2n-2p)) = (p-n)(t+2p+2). $$

We have that $\Ext^{2(n-p)+1}_S(S/\Pf^t, S)$ is cyclic if and only if $\lambda=(2(p-n)^{2n+1})$ is the minimal weight. By (\ref{ext_GL_decomposition_skew}) this happens if and only if $t=2(n-p)-1$.

\end{proof}

Finally, by Lemma \ref{minimalsocle} and Theorem \ref{degrees_gens_ext_skew}, we deduce the socle degrees of local cohomology. For the statement we note that dimension of $S=\Sym(\bigwedge^2 \CC^{m})$ is $\binom{m}{2}=2n^2+n$. 
\begin{corollary}
    The minimal degree of generators of $\Soc(H^{2n^2-n-2p-1}_\frakm(S/\Pf^t))$ is $(n-p)(t+2p+2) - (2n^2+n)$. 
\end{corollary} 

We carry out an example in the skew-symmetric case.

\begin{example}\label{Ex2}
Let $m=5$ and $p=1$, in which case we are considering cohomological degree $3$ of $\Ext$. In this case, $n=2$ and $\Ext^3_S(S/\Pf^t, S)$ has infinite dimension. By \cite[Theorem 5.3]{RaicuWeymanWitt} we have
$$\Ext^3_S(S/\Pf^t) \neq 0 \quad \iff \quad t\geq 1,$$ and there is an infinite chain
$$
\Ext^3_S(S/\Pf, S)\subseteq \Ext^3_S(S/\Pf^2, S)\subseteq \Ext^3_S(S/\Pf^3, S)\subseteq  \cdots,
$$
and there is a decomposition:
\begin{align*}
\Ext^3_S(S/\Pf^4, S) & =  \left( \bigoplus_{k\geq l\geq -2 \geq h \geq -5} \Schur_{(k,l,-2,h, -5)} \CC^5\right)\oplus \left( \bigoplus_{k\geq l\geq -2 \geq h \geq -4} \Schur_{(k,l,-2,h, -4)} \CC^5\right) \\
& \oplus \left(\bigoplus_{k\geq l\geq -2 \geq h \geq -3} \Schur_{(k,l,-2,h, -3)} \CC^5\right)\oplus \left( \bigoplus_{k\geq l\geq -2} \Schur_{(k,l,-2^3)} \CC^5\right).
\end{align*} 
The irreducible representation that generates $\Ext^3_S(S/\Pf^4, S)$ is $\Schur_{(-2^3,-5^2)}\CC^5$. The second, third and fourth direct summands are the decompositions of $\Ext^3_S(S/\Pf^3, S)$, which is generated by $\Schur_{(-2^3,-4^2)}\CC^5$. The third and fourth direct summands are the decompositions of $\Ext^3_S(S/\Pf^2, S)$, generated by $\Schur_{(-2^3, -3^2)}\CC^5$. Finally, the fourth direct summand is the decomposition of $\Ext^3_S(S/\Pf, S)$ generated by $\Schur_{(-2^5)}\CC^5$. In particular, it is cyclic.
\end{example}

\section*{Acknowledgments}
We thank Claudiu Raicu for his advice regarding the argument in Section \ref{genericModuleSection}, and we thank Wenliang Zhang for valuable conversations and suggestions. The first author was partially supported by NSF Grant No. DMS 1752081.

\end{document}